\documentclass[12pt]{amsart}
\usepackage{amsmath}
\usepackage[all]{xy}
\usepackage{cleveref}

\allowdisplaybreaks

\theoremstyle{plain}
\newtheorem{theorem}{Theorem}[section]
\newtheorem*{mthm*}{Main Theorem}

\newtheorem{lemma}[theorem]{Lemma}
\newtheorem{prop}[theorem]{Proposition}

\newtheorem{maintheorem}{Theorem}

\newenvironment{mthm}[1]
{\renewcommand{\themaintheorem}{#1}
\begin{maintheorem}}
{\end{maintheorem}}

\newtheorem{mainquestion}{Question}

\newenvironment{mqu}[1]
{\renewcommand{\themainquestion}{#1}
\begin{mainquestion}}
{\end{mainquestion}}

\newtheorem{mainconjecture}{Conjecture}

\newenvironment{mconj}[1]
{\renewcommand{\themainconjecture}{#1}
\begin{mainconjecture}}
{\end{mainconjecture}}

\theoremstyle{definition}

\newtheorem{rem}[theorem]{Remark}
\newtheorem{dfx}[theorem]{Definition}

\newtheorem*{convx*}{Convention}

\def\RR{\mathbb{R}}

\def\HH{\mathbb{H}}
\def\AdS{\mathbb{A}\mathbbm{d}\mathbb{S}}

\def\tr{\mathrm{tr}}

\def\AdS{\mathbb{A}\mathrm{d}\mathbb{S}}

\def\dis{\displaystyle}

\def\Hess{\mathrm{Hess}}

\def\vol{\mathrm{vol}}

\def\rar{\rightarrow}

\def\lra{\longrightarrow}

\def\arr#1#2{\stackrel{#1}{#2}}

\def\grad{\mathrm{grad}}

\def\Proj{\mathbf{P}}
\def\ProjR{\RR\Proj}
\def\PO{\mathrm{PO}}

\def\O{\mathrm{O}}
\def\Id{\mathrm{Id}}
\def\II{\mathrm{I\!I}}
\def\Tr{\mathrm{Tr}}
\def\Vol{\mathrm{Vol}}

\renewcommand{\d}{\mathrm{d}}

\title[Minimizing volume of AdS $3$-manifolds]{Minimizing the volume of globally hyperbolic anti-de Sitter $3$-manifolds}

\author{Gabriele Mondello}
\address{Gabriele Mondello: Sapienza Universit\`a di Roma, Department of Mathematics ``Guido Castelnuovo'' - piazzale Aldo Moro 5, Roma 00185 Italy}
\email{mondello@mat.uniroma1.it}

\author{Nicolas Tholozan}
\address{Nicolas Tholozan: CNRS, ENS-PSL, 45 rue d’Ulm, 75005 Paris, France}
\email{Nicolas.Tholozan@ens.fr}

\keywords{Anti-de Sitter 3-manifolds, constant mean curvature surfaces, foliations, rigidity}

\subjclass{53C50, 53C42, 53C24}

\begin{document}

\begin{abstract}
    In this paper we show that the volume of a maximal globally hyperbolic Cauchy-compact anti-de Sitter $3$-manifold $M$ is at least $\pi^2|\chi(M)|$, and that this minimum value is attained if and only if $M$ is Fuchsian.
\end{abstract}

\maketitle

\section{Introduction}\label{sec:intro}

An \emph{anti-de Sitter manifold} $M$ of dimension $d+1$ is a Lorentzian manifold of constant sectional curvature $-1$. Up to passing to a covering of degree $4$, one can assume that $M$ is \emph{space and time orientable}, meaning that there is a consistent orientation on every timelike curve and every spacelike hypersurface on $M$. The manifold $M$ is then locally modelled (in the sense of Thurston) on the homogeneous space 
\[\AdS^{d+1} := \{[\mathbf x] \in \ProjR^{d+1} \mid x_1^2 + \ldots + x_d^2 - x_{d+1}^2 - x_{d+2}^2 < 0\}\]
equipped with the transitive action of $\PO_0(d,2)$, the identity component of $\O(d,2)/\{\pm \Id_{d+2}\}$.

The manifold $M$ is called \emph{Globally Hyperbolic Cauchy-compact} (later abbreviated as GHC) if it contains a closed spacelike hypersurface $\mathcal H$ (called a \emph{Cauchy hypersurface}) intersecting every inextensible timelike curve at a single point. It follows that $M$ is diffeomorphic to $\mathcal H \times \RR$. Such a GHC anti-de Sitter manifold is \emph{maximal} (later abbreviated as MGHC) if it is not properly contained in another GHC anti-de Sitter manifold of the same dimension.

In dimension $3$, MGHC anti-de Sitter manifolds were described precisely by the foundational work of Mess \cite{Mess}. They have the form
\[\Gamma \backslash \Omega_\Gamma~,\]
where $\Gamma$ is the fundamental group of a Cauchy surface, conveniently embedded in $\PO_0(2,2)$, and $\Omega_\Gamma\subset \AdS^3$ is its \emph{domain of dependence} (see  \cite[Sections 5 and 7]{Mess}). When $\Gamma$ can be conjugated into $\PO_0(2,1)$, the manifold $\Gamma \backslash \Omega_\Gamma$ contains a \emph{totally geodesic} Cauchy surface and will be called \emph{Fuchsian}. 

The Lorentzian metric $g$ on $M$ induces a positive volume form $\vol_g$ and we define the \emph{volume of $M$} as
\[\Vol(M) := \int_{M} \vol_g~.\]
Despite $M$ being non-compact, this volume is always finite, and Question 4.2 of the influential survey \cite{OpenQuestionsAdS} asks whether it is minimized by Fuchsian manifolds. The purpose of the present paper is to answer this question positively.

\begin{mthm}{A} \label{thm: Main}
Let $M$ be a MGHC $\AdS$ $3$-manifold with Cauchy surface $\Sigma$ of negative Euler characteristic $\chi(\Sigma)$. Then 
\[\Vol(M) \geq \pi^2 \vert \chi(\Sigma)\vert\]
with equality if and only if $M$ is Fuchsian.
\end{mthm}

The fine geometric properties of MGHC $\AdS$ $3$-manifolds have been thoroughly studied (see 
\cite{notes,BB,KS} among others). In particular Bonsante, Seppi and Tamburelli investigate their volume in \cite{BST} and describe precisely the coarse behaviour of the volume function on the moduli space of closed anti-de Sitter $3$-manifolds. However, Theorem \ref{thm: Main} does not seem to follow from their work. 

Our main result can be seen as a comparison between competing variational phenomena:
\begin{itemize}
\item Every MGHC $\AdS$ $3$-manifold contains a Cauchy surface of maximal area and a timelike curve of maximal length. Fuchsian manifolds \emph{minimize} the area of the maximal Cauchy surface
(see for instance \cite[Theorem 3.17]{KS}, or Lemma \ref{lm: AhlforsSchwarzPick} below)
but \emph{maximize} the length of a maximal timelike curve
(see, for instance, \cite[Corollary 6.25]{BB}).

\item Every MGHC $\AdS$ $3$-manifold $M$ decomposes as $M=\mathcal C_M\sqcup M_- \sqcup M_+$, where $\mathcal C_M$ is the \emph{convex core} of $M$ and $M_-$, $M_+$ are respectively the past and future of $\mathcal C_M$. If $M$ is Fuchsian, the convex core is reduced to the totally geodesic Cauchy surface and thus has volume $0$. Hence Fuchsian manifolds \emph{minimize} the volume of the convex core. On the other hand, they \emph{maximize} the volume of $M_+$ and $M_-$ \cite[Proposition G]{BST}.
\end{itemize}

To conclude this introduction, let us mention that the volume of higher dimensional MGHC manifolds has been investigated in \cite{MV}. Yet, the question whether Theorem \ref{thm: Main} holds in higher dimension is still wide open, leading to the following conjecture:

\begin{mconj}{B}
Let $M$ be a MGHC $\AdS$ manifold of dimension $d+1\geq 4$ whose Cauchy hypersurface $\mathcal H$ is homeomorphic to a closed hyperbolic manifold $\Gamma \backslash \HH^d$. Then
\[\Vol(M) \geq \frac{\pi}{2} \Vol(\Gamma \backslash \HH^d)~,\]
with equality if and only if $M$ is Fuchsian.
\end{mconj}

Finally, MGHC $\AdS$ manifolds whose Cauchy hypersurface is not homeomorphic to a locally symmetric space have been constructed in \cite{LeeMarquis} (in dimension $5\leq d+1 \leq 9$) and \cite{MST} (in all dimensions $\geq 5$). For such manifolds, it is not even clear what the minimizer of the volume should be.

\begin{mqu}{C}
Let $C$ be a closed manifold. What is the value of 
\[\inf\ \left\{\,\Vol(M) \ \big|\ \textrm{$M$ MGHC $\AdS$ manifold homeomorphic to $C\times \RR$} \,\right\}~?\]
\end{mqu}

\subsection*{Structure of the paper}

In \Cref{sec:MGHC} we recall some background on MGHC $\AdS$ $3$-manifolds, and in particular on their foliation by \emph{constant mean curvature surfaces} constructed by Barbot--B\'eguin--Zeghib \cite{BBZ}. In \Cref{sec:fuchsian} we prove a maximum principle for the volume form of a MGHC $\AdS$ $3$-manifold in ``CMC coordinates'' and 
we deduce Theorem \ref{thm: Main} from it.

\subsection{Acknowledgements}
We thank the anonymous referee for carefully reading the paper
and for useful remarks that improved the exposition.

The first-named author was partially supported by
INdAM GNSAGA research group,
by PRIN 2022 research project ``Moduli spaces and special varieties'',
and by Sapienza research projects ``Algebraic and differential aspects of varieties and moduli spaces'' (2022),
``Aspects of modular forms, moduli problems, applications to L-values'' (2023),
``Global, local and infinitesimal aspects of moduli spaces'' (2024).

\section{MGHC $\AdS$ $3$-manifolds and their CMC foliation}\label{sec:MGHC}

In this section we set up the notations and recall 
the foliation of MGHC $\AdS$ $3$-manifolds by constant mean curvature Cauchy surfaces
and the semi-linear differential equation satisfied by the associated lapse function $\ell_\tau$. We then describe the CMC foliation more explicitly in the Fuchsian case.\\

\smallskip

Let $(M,g)$ be a space and time oriented MGHC anti-de Sitter $3$-manifold and $\Sigma\subset M$ a (smooth) compact Cauchy surface of genus at least $2$.
We recall from \Cref{sec:intro} that, by \cite{geroch}, the manifold $M$ is diffeomorphic to $\Sigma\times\RR$.

Let $\nabla$ denote the Levi--Civita connection of $g$. Denote by $g_\Sigma$ the Riemannian metric induced by $g$ on $\Sigma$ and $N_\Sigma$ the future-pointing unit normal vector to $\Sigma$. 
If $f$ is a (smooth) function on $M$, we denote by $\grad(f)$ the {\it{gradient}} vector field
associated to $f$ with respect to the metric $g$.

\subsection{Constant mean curvature foliations} \label{ss: CMC foliation}

Recall that
the \emph{second fundamental form} $\II_\Sigma$ and the \emph{shape operator} $B_\Sigma$ of $\Sigma$ are respectively defined by
\[\II_\Sigma(X,Y) = g(\nabla_X Y, N_\Sigma)\]
and
\[B_\Sigma(X) =  -\nabla_X N_\Sigma\]
for any vector fields $X$ and $Y$ tangent to $\Sigma$.

\begin{rem}
    We warn the reader that, in the present case of $\Sigma$ a spacelike
    hypersurface in a Lorentzian manifold $M$,
    some people prefer to adopt the convention 
    $\II_\Sigma(X,Y)=-g(\nabla_X Y,N_\Sigma)$ and
    $B_\Sigma(X)=+\nabla_X N_\Sigma$. 
\end{rem}

The second fundamental form is a symmetric $2$-tensor on $T \Sigma$, while the shape operator is a field endomorphisms of $T\Sigma$ that is dual to $\II_\Sigma$ with respect to $g_\Sigma$, i.e.
\[g_\Sigma(B_\Sigma(X),Y) = \II_\Sigma(X,Y)~.\]
The principal curvatures of $\Sigma$ are the eigenvalues of $B_\Sigma$ and the \emph{mean curvature} of $\Sigma$ is their average, namely half the {\it{trace}} of $B_\Sigma$. 

\begin{dfx}
The surface $\Sigma$ has \emph{constant mean curvature $\tau \in \RR$} if 
\[\frac{1}{2}\Tr(B_\Sigma)\equiv \tau\]
at every point of $\Sigma$.
\end{dfx}

In the particular case $\tau=0$, a surface $\Sigma$ whose mean curvature vanishes identically maximizes its area among all Cauchy surfaces, and is called a \emph{maximal surface}.\\

The Gauss curvature $K_\Sigma$ of $\Sigma$, on the other hand, is related to the \emph{determinant} of $B_\Sigma$ by \emph{Gauss's equation}:
\begin{equation} \label{eq: Gauss}
K_\Sigma = -1 - \det(B_\Sigma)~.
\end{equation}
In particular, when $\Sigma$ has constant mean curvature $\tau$, one easily gets the following inequality:

\begin{prop}\label[proposition]{gaussian}
If $\Sigma$ is a Cauchy surface of constant mean curvature $\tau$ inside $M$, then
\[K_\Sigma \geq -(1+\tau^2)~.\]
Moreover the equality is everywhere attained if and only if $\Sigma$ is \emph{umbilical}, in which case $M$ is Fuchsian.
\end{prop}

Here, a surface $\Sigma$ is called \emph{umbilical} if $\II_\Sigma$ is a scalar multiple of $g_\Sigma$, or equivalently if $B_\Sigma$ is a scalar multiple of the identity.

\begin{proof}[Proof of Proposition \ref{gaussian}]
With respect to a local orthonormal frame $(e_1,e_2)$
on $\Sigma_\tau$, we can write $B_\Sigma=(b_{ij})$ for suitable local functions $b_{11}$, $b_{12}$ and $b_{22}$. Thus we have 
\[
K_\Sigma=-1-(b_{11}b_{22}-b_{12}^2)\, ,
\qquad \tau=\frac{1}{2}(b_{11}+b_{22})\, .
\]
It follows that
\[
K_\Sigma+1+\tau^2=\frac{1}{4}(b_{11}-b_{22})^2+b_{12}^2\geq 0\, .
\]
Hence $K_\Sigma\equiv -(1+\tau^2)$ everywhere if and only if $b_{11}\equiv b_{22}$ and $b_{12}\equiv 0$ everywhere, namely $\Sigma$ is umbilical. In that case, a constant speed normal flow can be used to deform the umbilical 
subsurface $\Sigma$ into a totally geodesic Cauchy surface, thus showing that $M$ is Fuchsian.    
\end{proof}

Barbot, B\'eguin and Zeghib proved that constant mean curvature surfaces always exist and form a foliation of $M$.
\begin{theorem}[Barbot--B\'eguin--Zeghib \cite{BBZ}]
Let $M$ be a MGHC $\AdS$ $3$-manifold. Then there exists a proper smooth submersion $m: M\to \RR$ such that, for all $\tau \in \RR$, $m^{-1}(\tau)$ is a Cauchy surface of constant mean curvature $\tau$. Moreover, $m^{-1}(\tau)$ is the unique Cauchy surface of constant mean curvature $\tau$.
\end{theorem}

From now on, we will denote by $\Sigma_\tau = m^{-1}(\tau)$ the unique Cauchy surface of CMC curvature $\tau$, by $g_\tau$ the restriction of $g$ to $\Sigma_\tau$, by $\vol_\tau$ the induced area form on $\Sigma_\tau$ and by $N_\tau$ its future-pointing unit normal. We also denote respectively by $\II_\tau$, $B_\tau$, and $K_\tau$ the second fundamental form, shape operator and Gauss curvature of $\Sigma_\tau$. Let $\nu$ be the timelike, normal vector field on $M$ given by
\[\nu := \frac{1}{g(\grad(m), \grad(m))} \grad(m)~,\]
which is constructed so that $\d m(\nu) = 1$.

Note that the flow $\phi^\nu$ generated by the vector field $\nu$
determines a global trivialization
\[
\Phi:
\xymatrix@R=0in{
\Sigma_0\times\RR \ar[r]^{\quad\cong} & M\\
(x,\tau) \ar@{|->}[r] & \phi^\nu_\tau(x)
}
\]
which maps $\Sigma_0 \times \{\tau \}$ diffeomorphically onto $\Sigma_\tau$.\\

The {\it{lapse function}} $\ell:M\rar\RR_{>0}$, defined as
\[\ell: = \sqrt{- g(\nu,\nu)} = \frac{1}{\sqrt{-g(\grad(m),\grad(m))}}\]
will play an important role in \Cref{sec:fuchsian}.

Denote by $\nu_\tau$ and $\ell_\tau$ the restrictions of $\nu$ and $\ell$ to $\Sigma_\tau$, so that
\[\nu_\tau = \ell_\tau N_\tau~.\]
The following property of $\ell_\tau$ is well-known
(see, for instance, \cite[Section 3]{AMT}). Here we include a complete computation for the reader's convenience.

\begin{prop}\label[proposition]{prop:lapse}
  The function $\ell_\tau:\Sigma_\tau\rar\RR_{>0}$ satisfies the equation  
  \begin{equation}\label{eq:cmc}
   \frac{1}{2} \Delta_{g_\tau}\ell_\tau-(2 \tau^2+K_\tau+2)\ell_\tau+1=0
\end{equation}  
where $\Delta_{g_\tau}$ denotes the standard
Laplacian $\Delta_{g_\tau}:=\Tr_{g_\tau} \mathrm{Hess}$ on $\Sigma_\tau$
with respect to the background metric $g_\tau$.
\end{prop}


\begin{proof} 
Fix $\tau$. Let $(X_1,X_2)$ be a local orthonormal frame tangent to $\Sigma_\tau$
and extend it over $M$ along the flow generated by $\nu$,
so that $0=[X_i,\nu]=\nabla_{X_i}\nu-\nabla_\nu X_i$. Let $(g_{ij})$ and $(b_i^j)$ be defined by $g_{ij}=g(X_i,X_j)$ and $BX_i=\sum_j b_i^j X_j$, where $B$ at
a point $p$ of $M$ is the shape operator (with respect to $N$) of the space-like
surface $\Sigma_{m(p)}$ at $p$.\\

{\bf{Aim.}}
We will compute the derivative of $\sum_i g(BX_i,X_i)$ along $\nu$ on $\Sigma_\tau$ in two ways.\\

{\bf{Preliminary calculations.}}
Since $\nu=\ell N$, we have $\nabla_{X_i}\nu=\nabla_{X_i}(\ell N)=(X_i\cdot \ell)N+\ell\nabla_{X_i}N$ and so 
\[
\nabla_\nu X_i=\nabla_{X_i} \nu=
(X_i \cdot \ell) N-\ell BX_i\,.
\]
Then we remark that
$2g(\nabla_{\nu} N,N)=\nu\cdot g(N,N)=0$ and that
\begin{eqnarray*}
g(\nabla_\nu N,X_i)&=&\nu \cdot g(N,X_i)-g(N,\nabla_\nu X_i)=g(N,\nabla_{X_i}\nu)\\
&=&-g(N,(X_i\cdot \ell) N-\ell BX_i)=X_i\cdot \ell~,
\end{eqnarray*}
which implies that 
\[
\nabla_\nu N=\grad(\ell)\, .
\]
\smallskip

{\bf{First way.}}
Note that $g(BX_i,X_i)=g(\sum_j b_i^j X_j,X_i)=\sum_j g_{ij} b_i^j$ and that
$\nu\cdot (g_{ij}b_i^j)=(\nu\cdot g_{ij}) b_i^j+g_{ij}(\nu\cdot b_i^j)$.
Now,
\begin{eqnarray*}
    \nu\cdot g_{ij}&=&\nu\cdot g(X_i,X_j)
    =g(\nabla_\nu X_i,X_j)+g(X_i,\nabla_\nu X_j)\\
    &=& g((X_i\cdot \ell)N-\ell BX_i,X_j)+g(X_i,(X_j\cdot\ell)N-\ell BX_j)\\
&=&-2\ell g(BX_i,X_j)=-2\ell \sum_k b_i^k g_{jk}~.
\end{eqnarray*}
Thus
\[\nu\cdot g(BX_i,X_i)= -2\ell_\tau \sum_{j,k} b_i^k g_{jk} b^j_i+g_{ij}(\nu\cdot b_i^j)~.\]

Since $g_{ij}\big|_{\Sigma_\tau}=\delta_{ij}$ and
$\sum_i\nu\cdot b_i^i=\nu\cdot \tr(B)=2\nu \cdot m=2$, it
follows that
\begin{equation}\tag{$W_1$}\label{w1}
\nu\cdot \sum_i g(BX_i,X_i)\Big|_{\Sigma_\tau}=-2\ell_\tau \, \tr(B_\tau^2)+2
\end{equation}
holds.\\

{\bf{Second way.}}
Consider
\[
\nu\cdot g(BX_i,X_i)=g(\nabla_\nu (BX_i),X_i)+g(BX_i,\nabla_\nu X_i).
\]
As for the first summand, 
\[\nabla_\nu (BX_i)=-\nabla_\nu \nabla_{X_i}N=
-R(\nu,X_i)N-\nabla_{X_i}\nabla_\nu N~.\]
Since $R(N,X_i)N=-X_i$, we obtain that 
\[\nabla_\nu(BX_i)=
\ell X_i-\Hess(\ell)X_i~.\]
The first summand can now be written as
\[
\sum_i g(\nabla_\nu(BX_i),X_i)\Big|_{\Sigma_\tau}=2\ell_\tau-\Delta_{g_\tau}\ell_\tau
\]
As for the second summand, we have
$\sum_i g(BX_i,\nabla_\nu X_i)=\sum_i g(BX_i,(X_i\cdot\ell)N-\ell BX_i)=-\ell \,\tr(B^2)$.
We thus obtain 
\begin{equation}\tag{$W_2$}\label{w2}
\nu\cdot \sum_i g(BX_i,X_i)\Big|_{\Sigma_\tau}=2\ell_\tau-\Delta_{g_\tau}\ell_\tau-\ell_\tau\,\tr(B_\tau^2)
\end{equation}
along $\Sigma_\tau$.\\

{\bf{Equating \eqref{w1} and \eqref{w2}.}}
Equating the two expressions for $\nu\cdot\sum_i g(BX_i,X_i)\big|_{\Sigma_\tau}$, we obtain
\[
2\ell_\tau-\Delta_{g_\tau}\ell_\tau-\ell_\tau\,\tr(B_\tau^2)=-2\ell_\tau\,
\tr(B_\tau^2)+2
\]
which can be rewritten as
$\frac{1}{2}\Delta_{g_\tau}\ell_\tau-\frac{1}{2}\ell_\tau\,\tr(B_\tau^2)-\ell+1=0$.
Recalling that \[\tr(B_\tau^2)=\tr(B_\tau)^2-2\det(B_\tau)=(2\tau)^2-2(-1-K_\tau)=4\tau^2+2+2K_\tau~,\] we obtain the wished equation
\[
\frac{1}{2}\Delta_{g_\tau}\ell_\tau -(2\tau^2+K_\tau+2)\ell_\tau+1=0
\]
for $\ell_\tau$.
\end{proof}


%
%
%
%
%
%
%
%

\subsection{The Fuchsian case}

Consider a Fuchsian anti-de Sitter $3$-manifold $(M,g)$ and let $(\Sigma_0,g_0)$ be the totally
geodesic space-like Cauchy surface in $M$.

Using a normal flow,
it is easy to see that the universal cover $\widetilde{M}$ of $M$ is isometric to
$\HH^2\times \left(-\frac{\pi}{2},\frac{\pi}{2}\right)_{\tilde{r}}$ endowed with the Lorentzian metric $\cos^2(\tilde{r})g_{\HH^2}-d\tilde{r}^2$,
that $\pi_1(M)$ acts trivially on 
the second factor of $\HH^2\times\left(-\frac{\pi}{2},\frac{\pi}{2}\right)$, and that
the preimage of $\Sigma_0$ inside $\widetilde{M}$
is exactly $\HH^2\times\{0\}$. We then obtain a diffeomorphism
\[
(\Pi,r):M\arr{\sim}{\lra} \Sigma_0\times \left(-\frac{\pi}{2},\frac{\pi}{2}\right)\, ,
\]
where $\Pi$ is the closest-point projection onto $\Sigma_0$ and $r$ is 
the signed-distance from $\Sigma_0$ induced by $\tilde{r}$.

Since the space-like surface $\Sigma(r)$ in $M$ at distance $r$ from $\Sigma_0$
has intrinsic metric $g_r=\cos^2(r)g_0$ and unit normal vector $N=\partial_r$,
the following is just a computation.

\begin{lemma}
    Let $M$ be an MGHC Fuchsian $3$-manifold with totally geodesic Cauchy surface $\Sigma_0$, and let $\Sigma(r)$ be the space-like surface in $M$ that sits 
    at constant (signed) distance $r$ from $\Sigma_0$.   
Then for every $r\in\left(-\frac{\pi}{2},\frac{\pi}{2}\right)$ the surface $\Sigma(r)$
is umbilical and has constant mean curvature $\tau=\tan(r)$.
    Moreover $\Sigma(r)=\Sigma_\tau$ has
    curvature $K_\tau=-1/\cos^2(r)=-(1+\tau^2)$, and the lapse function satisfies $\ell_\tau=\cos^2(r)=1/(1+\tau^2)$.
\end{lemma}
\begin{proof}
Let $X$ be a vector field tangent to $\Sigma_0$ and extend it over $M$ so that 
$[\partial_r,X]=0$. We will compute $\partial_r g_r(X,X)$ in two ways.

First we have
\[
\partial_r g_r(X,X) =\partial_r \cos^2(r) \cdot g_0(X,X)=-2\cos(r)\sin(r) g_0(X,X)\, .
\]
On the other hand
\[
\partial_r g_r(X,X) =2 g_r(\nabla_{\partial_r}X,X)=2\cos^2(r)g_0(\nabla_X\partial_r,X)=
-2\cos^2(r)g_0(B_{\Sigma(r)}X,X)\, .
\]
Since $X$ is arbitrary, we conclude that $B_{\Sigma(r)}=\tan(r)\mathrm{Id}$.
Hence $\Sigma(r)$ is umbilical, with constant mean curvature $\tau=\tan(r)$ and constant Gauss curvature $-(1+\tau^2) = -1/\cos^2(r)$.

Finally, since, $\nu\cdot \tau = 1$, we have $1/\ell_\tau= N\cdot \tau =\partial_r\tan(r)=1/\cos^2(r)$. Hence $\ell_\tau=\cos^2(r)=1/(1+\tau^2)$.
\end{proof}

\section{Comparison with the Fuchsian case}\label{sec:fuchsian}

In this section we prove our main result, pursuing the following strategy.
We consider an MGHC $\AdS$ $3$-manifold $(M,g)$ and we define a different metric
$g'$ on $M$ that ``mimicks'' the metric of an MGHC Fuchsian manifold.
In particular $g'$ has constant Gaussian curvature equal to 
$K'_\tau:= -(1+\tau^2)$
along the Cauchy surface $\Sigma_\tau$
of constant mean curvature $\tau$
and has volume equal to the volume $\pi^2|\chi(\Sigma_0)|$ of a Fuchsian manifold. 
Then we use a maximum principle to show 
(Proposition \ref{inequality}) that the volume forms of $g$ and $g'$ satisfy $\vol_g\geq\vol_{g'}$, thus leading to the wished conclusion.\\

\smallskip

From now on, we let $M$ be a space and time-oriented MGHC $\AdS$ $3$-manifold whose Cauchy hypersurface has genus at least $2$. We let $\Sigma_\tau$ be the unique Cauchy surface of constant mean curvature $\tau$ and let $g_\tau$, $K_\tau$, $\II_\tau$, $B_\tau$, $\nu_\tau$, $\ell_\tau$, be as in \Cref{sec:MGHC}.\\

By the uniformisation theorem (\cite{koebe1,koebe2,poincare})
there exists a unique metric $h_\tau$ on $\Sigma_\tau$
that is conformal to $g_\tau$ and has constant Gaussian curvature $K'_\tau$.
Let $\sigma_\tau : \Sigma_\tau \to \RR_{>0}$ be the conformal factor such that
\[g_\tau=\sigma_\tau h_\tau~.\]
The equation for the variation of the curvature under conformal change gives:
\begin{equation}\label{eqK}
\frac{1}{2}
   \Delta_{g_\tau}\log(\sigma_\tau)=-K_\tau+\frac{1}{\sigma_\tau}K'_\tau
\end{equation}
Clearly, in the Fuchsian case $h_\tau\equiv g_\tau$ and $\sigma_\tau\equiv 1$. 

\begin{lemma} \label[lemma]{lm: AhlforsSchwarzPick}
The function $\sigma_\tau$ satisfies
    \begin{equation} \label{eq: AhlforsSchwarzPick}
\sigma_\tau \geq 1~.
\end{equation}
Moreover, if the equality is everywhere attained, then $M$ is Fuchsian.

\end{lemma}

It can be shown that, if the equality is attained {\it{at one point}}, then $M$ is Fuchsian.

\begin{proof}[Proof of Lemma \ref{lm: AhlforsSchwarzPick}]
Recall that $K_\tau \geq K'_{\tau}$ by Proposition \ref{gaussian}.
the inequality thus follows from a classical maximum principle
(see, for instance, \cite[Theorem 2']{yau} or the simpler \cite{wolpert}),
which is an elaboration of the so-called \emph{Ahlfors--Schwarz--Pick lemma}
(see \cite[Theorem A]{ahlfors}).

\end{proof}


The function $\ell_\tau$, on the other hand, obeys a similar maximum principle.
\begin{lemma}
    The function $\ell_\tau$ satisfies
    \begin{equation} \label{eq: inequality normal}
\ell_\tau \leq \ell_\tau' = \frac{1}{1+\tau^2}~.
\end{equation}
Moreover, if the equality is everywhere attained, then $M$ is Fuchsian.

\end{lemma}
\begin{proof}
    Recall that $\ell_\tau$ satisfies \Cref{eq:cmc}.
At a maximum point $\bar{p}\in\Sigma_\tau$ for $\ell_\tau$
we must have $(\Delta_{g_\tau}\ell_\tau)(\bar{p})\leq 0$ and so
\Cref{eq:cmc} implies that
\[
\ell_\tau(\bar{p})\leq \frac{1}{2\tau^2+K_\tau(\bar{p})+2}
\]
from which we obtain
\[
\ell_\tau\leq \frac{1}{2\tau^2+K_\tau(\bar{p})+2}\leq \frac{1}{1+\tau^2}\, .
\]
The last inequality follows from Proposition \ref{gaussian}, which also ensures that,
if equality is everywhere attained, then $M$ is Fuchsian.
\end{proof}

Inequalities \eqref{eq: AhlforsSchwarzPick} and \eqref{eq: inequality normal} are another occurrence of the ``competing phenomena'' mentioned in the introduction: Fuchsian manifolds tend to minimize the area of Cauchy hypersurfaces and maximize the timelike width. The key to the proof of the main theorem will be the following maximum principle:

\begin{prop}\label[proposition]{inequality}
On the Cauchy surface $\Sigma_\tau$ of constant mean curvature $\tau$ inside 
the MGHC anti-de Sitter manifold $M$
we have the uniform inequality
\[\sigma_\tau \ell_\tau \geq \ell_\tau' = \frac{1}{1+\tau^2}~.\]
Moreover, equality holds at every point if and only if $K_\tau \equiv K_\tau'$ and $\II_\tau \equiv \tau g_\tau$, in which case $M$ is Fuchsian.
\end{prop}

\begin{rem}
    Sharpening the maximum principle, one could actually prove that equality \emph{at one point} is enough to conclude that $K_\tau \equiv K_\tau'$.
\end{rem}

\begin{proof}[Proof of Proposition \ref{inequality}]
Combining \Cref{eqK} and \Cref{eq:cmc} we obtain
\[
   \frac{1}{2}
   \Delta_{g_\tau}\log(\sigma_\tau \ell_\tau)=
- K_\tau +\frac{1}{\sigma_\tau} K_{\tau}'+
(2\tau^2+K_\tau+2)-\frac{1}{\ell_\tau}-\frac{\|d\ell_\tau\|^2_{g_\tau}}{2 \ell_\tau^2}\,.
\]
Suppose that $\sigma_\tau \ell_\tau:\Sigma_\tau\rar\RR_{>0}$ achieves its minimum at the point $\bar{p}\in\Sigma_\tau$, so that $\Delta_{g_\tau}\log(\sigma_\tau \ell_\tau)(\bar{p})\geq 0$.
Then
\[
\frac{\vert K_{\tau}'\vert}{\sigma_\tau(\bar{p})}+\frac{1}{\ell_\tau(\bar{p})}\leq
2\tau^2 +2 -\frac{\|d\ell_\tau(\bar{p})\|^2_{g_\tau}}{2\ell_\tau^2(\bar{p})}\leq 2(\tau^2+1)=2 \vert K_\tau'\vert\, ,
\]
and, multiplying by $\ell'_\tau=1/|K'_\tau|$, we obtain
\begin{equation}\label{eq:2}
\frac{1}{\sigma_\tau(\bar{p})}+
\frac{\ell'_\tau}{\ell_\tau(\bar{p})} \leq 
2-\frac{\|d\ell_\tau(\bar{p})\|^2_{g_\tau}}{2\ell^2_\tau(\bar{p})}\ell'_\tau
\end{equation}
which is bounded above by $2$.
The inequality between arithmetic and geometric mean
then gives
\begin{equation*}
\sqrt{\dis
\frac{\ell'_\tau}{\sigma_\tau(\bar{p})\ell_\tau(\bar{p})}
}
\leq
\frac{1}{2}\left(
\frac{1}{\sigma_\tau(\bar{p})}+
\frac{\ell'_\tau}{\ell_\tau(\bar{p})}
\right)\leq 1~,
\end{equation*}
that is,
\[
\sigma_\tau(\bar{p}) \ell_\tau(\bar{p})\geq \ell'_\tau~.
\]
Since $\sigma_\tau\ell_\tau$ achieves its minimum at $\bar{p}\in\Sigma_\tau$, we obtain the wished inequality
\[
\sigma_\tau(p) \ell_\tau(p) \geq \ell'_\tau
\qquad\text{for all $p\in\Sigma_\tau$.}
\]

As for the last claim, assume now that the equality $\sigma_\tau\ell_\tau=\ell'_\tau$
holds at every point of $\Sigma_\tau$. Then
\eqref{eq:2} can be rewritten as
\begin{equation}\label{eq:arithmetic-geometric-mean}
1=\sqrt{\dis
\frac{\ell'_\tau}{\sigma_\tau\ell_\tau}
}\leq\frac{1}{2}\left(
\frac{1}{\sigma_\tau}+
\frac{\ell'_\tau}{\ell_\tau}
\right)\leq
1-\frac{\|d\ell_\tau\|^2_{g_\tau}}{4\ell_\tau^2}\ell'_\tau~.
\end{equation}
and we deduce that $d\ell_\tau\equiv 0$. Thus $\ell_\tau$ is constant and 
\[\sqrt{\dis
\frac{\ell'_\tau}{\sigma_\tau\ell_\tau}
} = \frac{1}{2}\left(
\frac{1}{\sigma_\tau}+
\frac{\ell'_\tau}{\ell_\tau}
\right)~,\]
which implies that $\sigma_\tau\equiv \ell_\tau/\ell'_\tau$ is constant too.
Together with our assumption
$\sigma_\tau \ell_\tau \equiv \ell_\tau'$,
we conclude that $\sigma_\tau\equiv 1$ everywhere.
As a consequence, $K_\tau\equiv K'_\tau$ and so $\Sigma_\tau$ is umbilical
and $M$ is Fuchsian by 
Lemma \ref{lm: AhlforsSchwarzPick} and
Proposition \ref{gaussian}.
\end{proof}

We can now conclude the proof of our main result.

\begin{proof}[Proof of Theorem \ref{thm: Main}]
Identify $\Sigma_0\times\RR$ and $M$ through
the global trivialization 
$\Phi: \Sigma_0 \times \RR \to M$ given by the CMC foliation,
introduced in \Cref{ss: CMC foliation}.

Through this identification, the anti-de Sitter metric of $M$ takes the form
\[g = g_\tau \oplus (-\ell_\tau^2 \ \d \tau^2)~.\]
Consider also the metric 
\[g'= h_\tau \oplus (-(\ell_\tau')^2\, \d \tau^2)~.\]
Since $g_\tau=\sigma_\tau h_\tau$, the associated volume forms satisfy
\[\vol_g = \frac{\sigma_\tau \ell_\tau}{\ell_\tau'} \vol_{g'}~.\]
Proposition \ref{inequality} ensures that $\sigma_\tau\ell_\tau\geq \ell'_\tau=\frac{1}{1+\tau^2}$, and so
\[\int_M \vol_g \geq \int_M \vol_{g'} = \int_\RR\left(\int_{\Sigma_\tau} \vol_{h_\tau} \right) \ell_\tau' \d \tau~.\]
Since $h_\tau$ has constant Gauss curvature $-(1+\tau^2)$, we get
\[\int_{\Sigma_\tau} \vol_{h_\tau} = \frac{2\pi \vert \chi(\Sigma_0)\vert}{1+\tau^2}~.\]
We thus conclude that
\[\int_M \vol_g \geq 2\pi \vert \chi(\Sigma_0)\vert  \int_\RR \frac{1}{(1+\tau^2)^2} \d \tau = \pi^2 \vert \chi(\Sigma_0)\vert~.\]

Equality holds only when $\sigma_\tau \ell_\tau \equiv \ell_\tau'$ everywhere on $M$. 
By Proposition \ref{inequality}, this implies that the maximal surface $\Sigma_0$ is umbilical, hence totally geodesic, and $M$ is thus Fuchsian.
\end{proof}

\bibliographystyle{alpha}

\end{document}